\RequirePackage[english]{babel}
\documentclass[a4paper]{amsart}
\title{Higher homotopy of groups definable in o-minimal structures} 

\author[A.~Berarducci]{Alessandro Berarducci}
\address{Universit\`a di Pisa, Dipartimento di Matematica, Largo Bruno Pontecorvo 5, 56127 Pisa, Italy}
\email{berardu@dm.unipi.it}
\thanks{First author is partially supported by GEOR MTM2005-02568}
\author[M.~Mamino]{Marcello Mamino}
\address{Classe di Scienze - Scuola Normale Superiore, Piazza dei Cavalieri, 7, 56126 Pisa, Italy}
\email{m.mamino@sns.it}
\author[M.~Otero]{Margarita Otero}
\address{Departamento de Matem\'aticas, Universidad Aut\'onoma de Madrid, Francisco Tom\'as y Valiente 7, 28049 Madrid, Spain}
\email{margarita.otero@uam.es}
\thanks{Third author is partially supported by GEOR MTM2005-02568 and Grupos UCM 910444}

\date{Sep. 29, 2008. Revised Feb. 10, 2009.}

\DeclareMathOperator{\R}{\mathbb R}

\DeclareMathOperator{\Q}{\mathbb Q}
\DeclareMathOperator{\Z}{\mathbb Z}

\newcommand{\M}{\mbox{$\mathcal{M}$}}

\theoremstyle{plain}
\newtheorem{theorem}{Theorem}
\newtheorem{lemma}[theorem]{Lemma}

\newtheorem{proposition}[theorem]{Proposition}
\newtheorem{corollary}[theorem]{Corollary}

\newtheorem{claim}{Claim}

\newtheorem{fact}[theorem]{Fact}
\theoremstyle{definition}
\newtheorem{remark}[theorem]{Remark}
\newtheorem{definition}[theorem]{Definition}
\newtheorem{example}[theorem]{Example}
\newtheorem{exercise}[theorem]{Exercise}

\numberwithin{theorem}{section}

\newcommand{\bt}{\begin{theorem}}
\newcommand{\et}{\end{theorem}}

\newcommand{\bl}{\begin{lemma}}
\newcommand{\el}{\end{lemma}}

\newcommand{\bfa}{\begin{fact}}
\newcommand{\efa}{\end{fact}}

\newcommand{\bexa}{\begin{example}}
\newcommand{\eexa}{\end{example}}

\newcommand{\bexe}{\begin{exercise}}
\newcommand{\eexe}{\end{exercise}}

\newcommand{\bprop}{\begin{proposition}}
\newcommand{\eprop}{\end{proposition}}

\newcommand{\bp}{\begin{proof}}
\newcommand{\ep}{\end{proof}}

\newcommand{\bc}{\begin{corollary}}
\newcommand{\ec}{\end{corollary}}

\newcommand{\bd}{\begin{definition}}
\newcommand{\ed}{\end{definition}}

\newcommand{\br}{\begin{remark}}
\newcommand{\er}{\end{remark}}

\def\NP{\mathbb{N^+}}

\def\struct{M}

\def\uc#1{\widetilde{#1}}

\def\closure#1{\operatorname{cl}\left({#1}\right)}

\def\ominpi#1#2{\pi_{#1}\left({#2}\right)}

\def\topi#1#2{\pi_{#1}\left({#2}\right)}
\def\H#1#2#3{\operatorname{H}_{#1}\left({#2};{#3}\right)}

\begin{document} 

\begin{abstract}
It is known that a definably compact group $G$ is an extension of a
compact Lie group $L$ by a divisible torsion-free normal subgroup. We show
that the  o-minimal higher homotopy groups of $G$ are isomorphic to the
corresponding higher homotopy groups of $L$. As a consequence, we obtain
that all abelian definably compact groups of a given dimension are
definably homotopy equivalent, and that their universal cover are
contractible.
\end{abstract}

\maketitle

\section{Introduction}
Definable groups in o-minimal structures have been studied by several
authors. 
The class of such groups includes all algebraic groups over either 
a real closed field or an algebraically closed field, but is not limited to such groups. 
Any compact Lie group is isomorphic to a real algebraic subgroup of the general linear group 
$GL(n,\R)$ and as such it also falls within the o-minimal category. On the other hand a non-compact Lie group may or may not be definable ({\em e.g.} the integers are not definable in any o-minimal structure). Working in the o-minimal category can have some advantages even if one  is only interested in the compact Lie groups. Indeed some proofs can be simplified due to the presence of a flexible set of tools like the cell decomposition theorem, the triangulation theorem (holding in the presence of field operations), and a model theoretic notion of ``Euler characteristic'' invariant under definable bijections, not necessarily continuous. The presence of these tools, combined with the general lack of analytical tools (like the exponential map of Lie group theory), has prompted some researchers in the field to adapt to the o-minimal category some of the usual devices of algebraic topology (homology, cohomology, homotopy) in order to investigate the definable groups. A fundamental result obtained in this way is that a definably compact abelian group $G$ of (o-minimal) dimension $d$ has the same torsion subgroup of a classical $d$-dimensional torus \cite{04EO}. By the work of several authors on the so called  ''Pillay's conjecture'' of \cite{04P} every definably compact group $G$ has a canonical type-definable divisible subgroup $G^{00}$ such that $G/G^{00}$ with the ``logic topology'' is a compact Lie group \cite{05BOPP}.  In the abelian case $G/G^{00}$ is a torus and its dimension can be computed by studying the structure of the torsion subgroup. This can be determined as follows. In \cite{08HPP} it is proved that in the abelian case $G^{00}$ is torsion free and therefore, since it is also divisible, $G$ and $G/G^{00}$ have isomorphic torsion subgroups. By the quoted result of \cite{04EO},  one then deduces that the o-minimal dimension of $G$ coincides with the dimension of $G/G^{00}$ as a Lie group. (This continues to hold also in the non-abelian case \cite{08HPP}. Also the fact that $G^{00}$ is divisible and torsion free can be extended to the non-abelian case \cite{07B}.)  
The proof in \cite{04EO} proceeds by studying the (o-minimal) fundamental group $\pi_1(G)$ of $G$, showing that it is isomorphic to $\Z^d$. In the light of Pillay's conjecture this can be equivalently expressed in the form \[\pi_1(G) \cong \pi_1(G/G^{00}) \eqno (1)\]
where the latter is the fundamental group of the Lie group $G/G^{00}$. The nice thing of this reformulation is that it makes sense to ask whether (1) continues to hold in the non-abelian case.
The paper \cite{07pB} deals with the corresponding question for the cohomology groups, showing that for every definably compact group $G$, every $i$, and any constant coefficient group,
\[H^i(G) \cong H^i(G/G^{00}) \eqno (2)\]
 where the left-hand side denotes the the $i$-th o-minimal (singular or sheaf) cohomology group.  
Taking $i=1$ we easily deduce the validity of (1) in the non-abelian case by the o-minimal Poincar\'e-Hurewicz theorem in \cite{04EO} and the corresponding classical result. We actually show in this paper that \[\pi_n(G) \cong \pi_n(G/G^{00}) \eqno (3)\] 
for all $n \geq 1$ and all definably compact groups $G$ (where the left-hand side is the o-minimal $n$-th homotopy group of $G$). The proof of (3) is non-trivial even in the abelian case, where it amounts to show that $$\pi_n(G) = 0\eqno (4)$$ 
for all $n>1$ (recalling that the case $n=1$ has already been dealt with). This would be easy to prove if one could show (by analogy with the torus) that $G$ factors definably into one-dimensional subgroups (as $\pi_n(G)$ would also factor), but in general this is not the case (here we measure the effect of the lack of the exponential maps). Since $\pi_n(G)$ can be proved to be divisible for $n>1$, to prove (4), it suffices to show that $\pi_n(G)$ is finitely generated. Clearly we cannot make a direct appeal to the triangulation theorem since a finite simplicial complex does not need to have a finitely generated $\pi_n$ (although it has finitely generated homology and cohomology groups). What we need is to use  the structure of o-minimal $H$-space of (the triangulated copy of) $G$. Recall that any topological group is an $H$-space and it is known that an $H$-space with finitely generated homology groups has finitely generated higher homotopy groups. This however holds in the classical setting, and we need to transfer the latter to the o-minimal category. To do so we exploit some results of \cite{08pBaO} which allow to transfer the  o-minimal $H$-space structure (but not the group structure) to the classical setting  -- so we can apply the above classical result --  and then go back to the o-minimal setting thanks again to a result in \cite{08pBaO} which establishes a group isomorphism between the classical and o-minimal higher homotopy groups of the relevant spaces.  As a consequence of (4),  we  first obtain that all the abelian definably compact definably connected groups of the same dimension are definably homotopy equivalent and therefore their o-minimal homology groups (over $\mathbb{Z}$) are torsion free. This is done using results of \cite{08pOP} and the o-minimal Whitehead theorem proved in \cite{08pBaO}. Then, we  apply results in \cite{08pBaO2} to  get that the universal covering group of our abelian definable group is contractible (in the  category of locally definable spaces).

Having proved (3) in the abelian case it remains to deal with the non abelian case. We have already discussed the case $n=1$. The case $n>1$ is 
a reduction to the abelian and to the semisimple centreless cases making use of the long exact homotopy sequence of a fibration and again results in \cite{08pBaO}. The whole of section 2 is devoted to prove that a surjective definable morphism $G \to G/H$ of definable groups is a definable fibration (in analogy with the corresponding result for Lie groups). Once we have a definable fibration we obtain a long exact homotopy sequence by \cite{08pBaO}. Of course to be able to compare $\pi_n(G)$ and $\pi_n(G/G^{00})$ we need to work with homotopy sequences both in the o-minimal and in the Lie category. To be able to match the two sides, we exploit the fact that the correspondence $G\mapsto G/G^{00}$ gives rise to a functor from definable groups (and definable homomorphisms) to compact Lie groups (and continuous homomorphisms) which, by \cite{07B}, preserves exact sequences. 

We assume the basics of o-minimality and o-minimal topology, as the reader may find in
\cite{98V}. For a general reference on definable groups see \cite{08O}, and for Lie groups see  \cite{04Bu}.

\subsection{Notations}

We work over an o-minimal expansion \M\ of a real closed field.  Definable
means definable in \M\  with parameters. We take the order topology in
$M$, the product topology in the products $M^n$, and the subspace topology
when working inside some subset of $M^n$. We denote by  $I$  the interval
$[0,1]\subset M$.  If $G$ is a definable group, $\pi_n(G)$ denotes the
o-minimal $n$-th homotopy group of the pointed space $(G, e)$ (see
\cite[section 4]{08pBaO}), where $e$ denotes the identity of $G$. If $G$
is a Lie group $\pi_n(G)$ is the classical $n$-th homotopy group of the
pointed space $(G, e)$, where $e$ denotes the identity of $G$.  Similar
notation for homology and cohomology.  Except in section 2,  \M\ is  a
sufficiently saturated structure.

\section{Definable fibre bundles }

\bd  Let $E$ and $B$ be definable sets. Let $p:E\to B$  be  a definable continuous map. 

We say that $p$ is a {\em definable fibre bundle}  with {\em  fibre} a definable set $F$ if there is a definable finite open definable covering of $B= \bigcup_{i=1}^k U_i$ and for each $i$, a definable homeomorphism $\varphi_i\colon U_i\times F\to p^{-1}(U_i)$ such that $p\circ\varphi_i\colon U_i\times F \to U_i$ is the projection onto the first factor.

 We say that $p$ is a  {\em definable fibration} if $p$ has the homotopy lifting property with respect to all definable sets {\em i.e.},  for every definable set $X$, for every definable homotopy $f:X\times I \to B$ and for every definable map $g: X\to E$ such that $p\circ g= f(-,0)$ there is a definable homotopy $h: X\times I\to E$ such that  $p\circ h=f$ ($h$ is a {\em lifting of} $f$) and $h(-,0)=g(-)$.
\ed

A definable covering map is an example of a definable fibre bundle in
which the fibre is finite, while a projection $p:B\times F\to B$ is an
example of a trivial fibre bundle with arbitrary  fibre. As the reader may
verify, both maps are also definable fibrations (the first one by uniqueness
of path lifting, see \cite[proposition 4.10]{08pBaO}). A less trivial
example of definable fibre bundle is provided by the following.

\begin{lemma}\label{quotienttobundle}
Let $H$ be a definable subgroup of a definable group $G$.
Then the natural map $p: G \to G/H$ is a definable fibre bundle
(with the quotient topology on $G/H$).
\end{lemma}
\begin{proof}
Consider a (discontinuous) definable section $s: G/H \to G$.
Let $X$ be a large subset of $G/H$ (where {\em large} means that that
$\operatorname{dim}((G/H) \setminus X) < \operatorname{dim}(G/H)$)
such that $s$ is continuous on $X$.
Without loss of generality, we can assume $X$ to be open in $G/H$, for,
observe that the interior of $X$ is large, since the dimension of the
boundary of $X$ is strictly less than the dimension of $G/H$.
By \cite[claim 2.12]{00PPS} 
finitely many left translates $X_1, \ldots , X_m$ of $X$ cover $G/H$.
Fix an  $i\leq m$.  We have an induced section $s_i: G/H \to G$ (a conjugate of $s$) 
that is continuous on $X_i$.  Therefore, the  map $\varphi_i\colon X_i \times H\to p^{-1}(X_i) $  defined by 
$ \varphi_i(gH,h) =s_i(gH)h$ is a  definable homeomorphism with its  inverse sending $g$ to $ (gH, (s_i(gH))^{-1}g)$. Finally, note that $p\circ \varphi_i=p_1$, where  $p_1 : X_i \times H \to X_i$ is the projection onto  the first coordinate. 
\end{proof}

This section will be devoted to prove the following.
\begin{theorem}\label{mainfib}
Every definable fibre bundle is a definable fibration.
\end{theorem}
Which we need, in particular, for the following.
\bc\label{fibration} Let $H$ be a definable  subgroup of a definable group
$G$. Then $p \colon G \to G/H$  is a definable  fibration
(with the quotient topology on $G/H$).
\ec

The proof of theorem \ref{mainfib} is an adaptation of section 7 in chapter 2 of
\cite{66S}. The main difference lies in that we can not work with the
path-space (which is not even a definable object) and its compact-open topology. 

{\em In the rest of this section, $p:E\to B$ is a
fixed continuous definable map between definable sets
$E$ and $B$}.

Now, we are going to consider a definable subset $X$ of the path-space of
$B$. More precisely, we will consider a {\em family} of paths
parametrized by some definable set $X$, which we will represent using a
continuous definable function from $X\times I$ to $B$. The reader may
think of any $x\in X$ as a path itself, and, by abuse of notation, of
$f(x,-)$ as $x(-)$. From this point of view $f$ is a continuous map from
$X$ to the path-space of $B$ with the compact-open topology.
For each such object, we will call {\em lifting function} any continuous
definable map sending each pair $(e,x)$, where $x$ is an element of $X$
and $e$ is a point in $p^{-1}(x(0))$, to a lifting of $x$ starting at $e$.
It is easily shown (lemma \ref{fibifflf}) that the existence of a lifting
function for any $X$ is equivalent to $p$ being a definable fibration.
Moreover, in lemma \ref{elfimplieslf}, we will show how, given lifting
functions for each element of a finite definable open covering of $X$, we
can get a lifting function for $X$ (basically blending the pieces via a
definable partition of unity).
Finally, through lemmata \ref{slicing_lemma} and \ref{piecetogether}, we
will use the local triviality of $p$ to get such a covering for any subset
$X$ of the pathspace of $B$.

\begin{definition}
For any continuous definable  function $f:X\times I\to B$, let
$\bar{B}_f$ be the definable set
\[
\bar{B}_f = \left\{ (e,x)\in E\times X  \;|\;  p(e)=f(x,0) \right\}.
\]
We will call a continuous definable 
function $\lambda:\bar{B}_f\times I\to E$ a {\em lifting
function for} $f$ if for all $(e,x)\in \bar{B}_f$,
\[
p\circ\lambda\left((e,x),-\right) = f(x,-)
\;\;\;\mbox { 
and
 }\;\;\;
\lambda\left((e,x),0\right) = e.
\]
\end{definition}

\begin{lemma}\label{fibifflf}
The map $p$ is a definable fibration if and only if for any definable set $X$
and for any continuous definable function $f:X\times I\to B$ there is a
lifting function for $f$.
\end{lemma}
\begin{proof}
To prove the {\em if} part,
consider functions $f:X\times I \to B$ and $g:X\to E$ such that
$f(-,0)=p\circ g(-)$. Let $\lambda$ be a lifting function for $f$. Then
$h=\lambda((g\times Id)(-),-):X\times I\to E$ is a lifting of $f$ and
$h(-,0)=g(-)$.

To prove the {\em only if} part,
consider the homotopy
$h:\bar{B}_f\times I\to B$ defined by $h((e,x),t)=f(x,t)$. Since $h(-,0) =
p\circ{p_1}$, where  ${p_1}:\bar{B}_f \to E$ is the projection on the first
component, we may lift $h$ to a map $\lambda:\bar{B}_f\times I\to
E$ such that $\lambda(-,0)={p_1}$, which is a lifting function for $f$.
\end{proof}

\begin{definition}
For any continuous definable  function $f:X\times I\to B$ and for any
definable subset $W\subset X$, let
$\widetilde{W}_f$ be the definable set
\[
\widetilde{W}_f = \left\{ (e,x,s)\in E\times W\times I  \;|\; 
	p(e) = f(x,s) \right\}.
\]
We will call a continuous definable  function
$\Lambda:\widetilde{W}_f\times I\to E$ an {\em extended lifting function for
$f$ over} $W$ if for all $(e,x,s)\in \widetilde{W}_f$,
\[
p\circ\Lambda\left((e,x,s),-\right) = f(x,-)
\;\;\;\mbox { 
and
 }\;\;\;
 \Lambda\left((e,x,s),s\right) = e.
\]
\end{definition}

The idea behind an extended lifting function is that it maps any triple
$(e,x,s)$, where $x$ is a path, $s$ is an element of $I$, and $e$ is a
point in $p^{-1}(x(s))$, to a lifting of $x$ passing through $e$ at time
$s$.  As we will see in the next lemma, an extended lifting function is
nothing but a lifting function in disguise (basically because, using
appropriate lifting functions, we can lift separately the two halves of
any path which are before and after $s$); in fact, extended lifting
functions were introduced for merely technical reasons.

\begin{lemma}\label{lfiffelf}
The following statements are equivalent:
\begin{enumerate}
\item[(i)] for any definable set $X$ and for any definable continuous function
$f:X\times I\to B$ there is lifting function for $f$;
\item[(ii)] for any definable set $X$ and for any definable continuous function
$f:X\times I\to B$ there is an extended lifting function for $f$ over $X$.
\end{enumerate} 
\end{lemma}
\begin{proof}
The direction $\text{(ii)}\implies\text{(i)}$ is trivial. For
$\text{(i)}\implies\text{(ii)}$,
let $Y=X\times I$ and let
$g',g'':Y\times I\to B$ be defined by
\begin{align*}
g'((x,s),t) &= \begin{cases}
f(x,s-t) &\text{for $t\le s$}\\
f(x,0) &\text{for $t > s$}
\end{cases} \\
g''((x,s),t) &= \begin{cases}
f(x,s+t) &\text{for $t \le 1-s$}\\
f(x,1) &\text{for $t > 1-s$}
\end{cases}
\end{align*}
By (i)
there are lifting functions $\lambda'$ and $\lambda''$ for
$g'$ and $g''$ respectively. Now, an extended lifting function
$\Lambda:\widetilde{X}_f\times I\to E$ for $f$ over $X$ is given by
\[
\Lambda((e,x,s),t) = \begin{cases}
\lambda'((e,(x,s)),s-t) &\text{for $t \le s$} \\
\lambda''((e,(x,s)),t-s) &\text{for $t > s$}
\end{cases}
\]
which is easily proved continuous, moreover for $t\le s$ we have
\[
p\circ\Lambda((e,x,s),t) =
p\circ\lambda'((e,(x,s)),s-t) =
g'((x,s),s-t) =
f(x,t)
\]
and similarly for $s<t$; and finally
\[
\Lambda((e,x,s),s) = \lambda'((e,(x,s)),0) = e.
\]
\end{proof}

\begin{lemma}\label{elfimplieslf}
Let $X$ be a definable set, and let $f:X\times I\to B$ be definable and
continuous. Suppose that there is a finite open covering $\mathcal{W}$ of
$X$ such that for each element $W$ of $\mathcal{W}$ there is an extended
lifting function for $f$. Then there is a lifting function for $f$.
\end{lemma}
\begin{proof}
Let $\mathcal{W}=\{W_j\}_{j\in J}$ and
consider a definable
partition of unity $\{g_j\}_{j\in J}$ on $X$
({\em i.e.} $g_j:X\to\struct^{\ge0}$ and $\sum_j g_j\equiv 1$)
such that
$\closure{W'_j}\subset W_j$ for each $j$, where
$W'_j=\{x\in X  \;|\;  g_j(x)>0\}$.
For each subset $\kappa$ of $J$ define
\begin{align*}
	W'_\kappa &= \bigcup_{j\in\kappa} W'_j &
	f_\kappa &= f|_{W'_\kappa\times I} &
	\bar{B}_\kappa &= \bar{B}_{f_\kappa} &
	g_\kappa &= \sum_{j\in\kappa}g_j 
\end{align*}
Consider a maximal subset $\alpha$ of $J$ such that there
is a lifting function $\lambda_\alpha$ for $f_\alpha$.
For the sake of contradiction, suppose there is a  $j_0\in J\setminus\alpha$. We shall construct a lifting
function for $f_\beta$, where $\beta=\alpha\cup \{j_0\}$.
Let $h=\frac{g_\alpha}{g_\beta}:W_\beta\to I$. Clearly $h$
is the constant $1$ on $W'_\alpha\setminus W'_{j_0}$,
the constant $0$ on $W'_{j_0}\setminus
W'_\alpha$, and takes values in $(0,1)$ on $W'_\alpha\cap W'_{j_0}$.
Define
$\mu:\bar{B}_{\{j_0\}}\to E$ by:
\[
\mu(e,x) =
\begin{cases}
e &\text{for $0 \le h(x) < \frac{1}{2}$} \\
\lambda_\alpha((e,x),2h(x)-1) &\text{for $\frac{1}{2} \le h(x) < 1$} \\
\end{cases}
\]
Consider an extended lifting function $\Lambda$ for $f$ over
$W_{j_0}$, then we claim that $\lambda_\beta:\bar{B}_\beta\times I\to E$ defined by:
\begin{multline*}
\lambda_\beta((e,x),t) =\\
=
\begin{cases}
\Lambda((e,x,0),t) &\text{for $0 \le h(x) < \frac{1}{2}$}\\
\lambda_\alpha((e,x),t)
&\text{for $\frac{1}{2} \le h(x) \le 1$ and $0 \le t \le 2h(x)-1$} \\
\Lambda((\mu(e,x),x,2h(x)-1),t) 
&\text{for $\frac{1}{2} \le h(x) \le 1$ and $2h(x)-1 < t \le 1$} \\
\end{cases}
\end{multline*}
is a lifting function for $f_\beta$.
In fact, to prove the continuity of $\lambda_\beta$ suffices to observe that:
\begin{align*}
\Lambda((e,x,0),t) &= \lambda_\alpha((e,x),t)
	&\text{when $h(x)=\frac{1}{2}$ and $t=0$} \\
\Lambda((e,x,0),t) &= \Lambda((\mu(e,x),x,2h(x)-1),t)
	&\text{when $h(x)=\frac{1}{2}$ and $t\ge 0$} \\
\lambda_\alpha((e,x),t) &= \Lambda((\mu(e,x),x,2h(x)-1),t)
	&\text{when $\frac{1}{2}\le h(x)$ and $t=2h(x)-1$}
\end{align*}
for $x\in\closure{W'_{j_0}}\subset W_{j_0}$. Moreover, the equations
$p\circ\lambda_\beta((e,x),t) = f(x,t)$
and $\lambda_\beta((e,x),0) = e$, which are required by the definition of
lifting function, hold because $\Lambda$ and $\lambda_\alpha$ are
(extended) lifting functions.
\end{proof}

\begin{lemma}\label{slicing_lemma}
Let $f$ be a  definable continuous function $f:X\times I\to B$ and let
$\mathcal{U}$
be a finite definable open covering of $B$. Then, there are continuous
definable
functions $0\equiv g_1\le\ldots\le g_{k+1}\equiv 1:X\to I$ such that:\\
($\star$) for each $x\in X$ and each $i\in\{1,\ldots,k\}$, 
the set $f(x,[g_i(x),g_{i+1}(x)])$ is entirely contained in some
element of $\mathcal{U}$.
\end{lemma}
\begin{proof}
{\em Caveat:} Observe that if ($\star$) holds for an {\em ordered} set of
definable functions
$\mathcal{G}$ then it also holds for any superset of $\mathcal{G}$.\\
The proof is by induction on the dimension of $X$.
If $\operatorname{dim}(X)=0$ then $X$ is finite and the
statement is trivial.

Suppose the statement is true for
$\operatorname{dim}(X)\le n$, then we will prove it
for $\operatorname{dim}(X)=n+1$.
Observe that for each $x\in X$ there are $0=l_1<\ldots<l_{h(x)+1}=1$ such that
for each $i\in\{1,\ldots,h(x)\}$ the set $f(x,[l_i,l_{i+1}])$ is entirely contained in some
element of $\mathcal{U}$.
Without loss of generality we may suppose $h(x)$ to be minimal for each $x$,
then $h=\operatorname{sup}_{x\in X} h(x) < \infty$ by
uniform finiteness.
By definable choice (and by the caveat) we have definable functions
$0\equiv l_1\le\ldots\le l_{h+1}\equiv 1: X\to I$,
not necessarily continuous, such that ($\star$) holds for them.
Let $\mathcal{X}$
be a finite partition of $X$ into definable sets
such that each of the functions $l_i$ is continuous on each of the
sets of $\mathcal{X}$. Now consider the set $\mathcal{X}'$ of the interiors
of the sets in $\mathcal{X}$, and let $Y=X\setminus\bigcup\mathcal{X}'$:
clearly $\operatorname{dim}(Y)\le n$. By the induction hypothesis applied
to $f_{|Y\times I}$ and $\mathcal{U}$, we get definable continuous functions
$0\equiv g_1\le\ldots\le g_{k+1}\equiv 1:Y\to I$ such that 
for each $x\in Y$ and each $i\in\{1,\ldots,k\}$
the set $f(x,[g_i(x),g_{i+1}(x)])$ is entirely contained in some
element of $\mathcal{U}$.
Consider a definable open neighbourhood $Z$ of $Y$ in $X$ such that there is a
retraction $r:Z\to Y$.
We claim that, by continuity of the functions $g_i\circ r$, we can
find in $Z$ a new definable open neighbourhood $Z'$ of $Y$ such that
for each $x\in Z'$ and each $i\in\{1,\ldots,k\}$
the set $f(x,[g_i\circ r(x),g_{i+1}\circ r(x)])$ is entirely contained in
some element of $\mathcal{U}$.
In fact, let $\mathcal{U}=\{U_j\}_j$: working in $Z$, we observe that,
since $I$ is closed and bounded, the projection $p_1:Z\times I\to Z$
on the first component is a closed function,
hence, for each $i$ and $j$, the set
\begin{align*}
Z_{i,j} &=
\{x\in Z \;|\; f(x,[g_i\circ r(x),g_{i+1}\circ r(x)])\subset
U_j\}\\
&= Z\setminus p_1\left(\left\{\left(x,t\right)\in Z\times I  \;|\;  t\in
\left[g_i\circ r(x),g_{i+1}\circ r(x)\right]\right\}
\setminus f^{-1}(U_j)\right)
\end{align*}
is open, and so is $Z'=\bigcup_{i,j}Z_{i,j}$, which contains $Y$ by
hypothesis.

Finally, we have a definable open covering
$\mathcal{X}''=\mathcal{X}'\cup\{Z'\}$ of $X$ such
that the statement of the lemma holds for each element of $\mathcal{X}''$:
we will show that this implies the statement for $X$.
For, let $\mathcal{X}''=\{V_1,\ldots,V_m\}$ and suppose that for each $i$
the statement holds for $V_i$ and definable functions $g_{i,j}:V_i\to I$.
Consider a shrinking $\{V'_1,\ldots,V'_m\}$ of $\mathcal{X}''$ such that
$\closure{V'_i}\subset V_i$ for each $i$.
Since $I$ is definably contractible, for each $i$ and $j$ we have
a continuous definable extension $g'_{i,j}:X\to I$ of
${g_{i,j}}_{|\closure{V'_i}}$. Hence, the continuous definable functions
\[
	g_r(x) = \text{ the $r$-th element of $\{g'_{i,j}(x)\}_{i,j}$ in
	ascending order }
\]
satisfy ($\star$), by the caveat, since a subset of them does on
each of the sets $\closure{V'_i}$, which cover $X$.
\end{proof}

\begin{lemma}\label{piecetogether}
Let $U_1,\ldots,U_k$ be definable subsets of $B$.
Suppose that, for any definable set $X$, for any definable continuous
function $f:X\times I\to B$, and for any $i$, there is an extended lifting
function for $f$ over $\{x\in X \;|\;  f(x,I)\subset U_i\}$.
Then, for any definable set $X$, for any definable continuous function
$f:X\times I\to B$, and for any definable continuous functions $0\equiv
g_1\le\ldots\le g_{k+1}\equiv 1:X\to I$, there is an extended lifting
function for $f$ over $W$, where
\[
W = \{ x\in X  \;|\;   f(x,[g_i(x),g_{i+1}(x)])
\subset U_i \mbox{ for all } i=1,\dots,k \}.
\]
\end{lemma}
\begin{proof}
The idea behind the following convoluted formul\ae\ is that, for any $i$, using
the extended lifting function for $U_i$ (which we have by hypothesis), we can
lift any path $x\in W$ restricted to the interval $[g_i(x),g_{i+1}(x)]$.
Hence, all we have to do is to lift such paths interval by interval, taking
care that they fit together properly.

Let $Y=\{(x,a,b)\in X\times I^2 \;|\; a<b\}$, and let $h:Y\times I\to B$ be defined by
\[
h((x,a,b),t) =
\begin{cases}
f(x,a) &\text{for $t \in [0,a)$} \\
f(x,t) &\text{for $t \in [a,b]$} \\
f(x,b) &\text{for $t \in (b,1]$}
\end{cases}
\]
By hypothesis, for each $i$, we have an extended lifting function
$\Lambda_i$ for $h$ over $Y_i=\{y\in Y \;|\;  h(y,I)\subset U_i\}$.
Let
\[
L_{i,j}=\{((e,x,s),t)\in\widetilde{W}_f\times I  \;|\; 
	s\in[g_i(x),g_{i+1}(x)] \text{ and }
	t\in[g_j(x),g_{j+1}(x)]\}
\]
and define $\Lambda'_{i,j}: L_{i,j}\to E$ by
\begin{multline*}
\Lambda'_{i,j}((e,x,s),t) =\\
=
\begin{cases}
	\Lambda_i((e,(x,g_i(x),g_{i+1}(x)),s),t)
	&\text{for $i=j$} \\
	\Lambda_j((\Lambda'_{i,j-1}((e,x,s),g_j(x)),(x,g_j(x),g_{j+1}(x)),g_j(x)),t)
	&\text{for $j>i$} \\
	\Lambda_j((\Lambda'_{i,j+1}((e,x,s),g_{j+1}(x)),(x,g_j(x),g_{j+1}(x)),g_{j+1}(x)),t)
	&\text{for $j<i$} 
\end{cases}
\end{multline*}
It is routine to prove simultaneously by induction on $|j-i|$ that these are
well-defined continuous functions and
\[
\tag{$\star$}
\label{star}
\forall ((e,x,s),t)\in L_{i,j} \quad f(x,t) = p\circ\Lambda'_{i,j}((e,x,s),t)
\]
Now, let
\begin{gather*}
l:X\times I \to \{1,\ldots,k\} \\
l(x,s) = i<k \text{ iff $g_i(x)<g_{i+1}(x)$ and $s \in [g_i(x),g_{i+1}(x))$} \\
l(x,s) = k \text{ iff $s \in [g_k(x),g_{k+1}(x)]$}
\end{gather*}
and define
\begin{gather*}
\Lambda':\widetilde{W}_f\times I \to E \\
\Lambda'((e,x,s),t) = \Lambda'_{l(x,s),l(x,t)}((e,x,s),t)
\end{gather*}
We have the continuity of $\Lambda'$
since
the functions $\Lambda'_{i,j}$ coincide on the intersections of their
domains, which are closed. That, in turn, may be proved easily
observing that
\[
\Lambda'_{i,i}((e,x,s),s) = e
\]
by definition,
\[
\Lambda'_{i,j}((e,x,g_i(x)),t) = \Lambda'_{i-1,j}((e,x,g_i(x)),t)
\]
by induction on $|j-i|$, and
\[
\Lambda'_{i,j}((e,x,s),g_j(x)) = \Lambda'_{i,j-1}((e,x,s),g_j(x))
\]
by definition.
Since $\Lambda'((e,x,s),s) = e$,  by (\ref{star}) we can conclude   that $\Lambda'$ is an extended lifting function for $f$ over $W$.
\end{proof}

\begin{theorem}\label{fibtheo}
Given a  definable map $p:E\to B$, if there is a finite definable open covering
$\mathcal{U}$
of $B$,
such that for each $U\in\mathcal{U}$
the map $p|_{p^{-1}(U)}:p^{-1}(U)\to U$ is
a definable fibration, then $p$ is a definable fibration.
\end{theorem}
\bp We  apply lemma \ref{fibifflf}. Fix a definable set $X$ and a continuous definable function $f:X\times I\to B$, we have to show that there is a lifting function for $f$ over $X$.  By lemma \ref{elfimplieslf}, it suffices to find a  definable finite open covering $\mathcal{W}$ of $X$  and  for each
$W\in\mathcal{W}$ an extended lifting function for $f$ over  $W$.

 On the other hand, by lemma \ref{slicing_lemma}
there are continuous functions $0\equiv g_1\le\ldots\le g_{k+1}\equiv 1:X\to I$ such
that for each $x\in X$ and each $1\leq i\leq k$ the set $f(x,[g_i(x),g_{i+1}(x)])$ is entirely contained in some element of $\mathcal{U}$.
Let $\mathcal{U}=\{U_j\}_{j=1}^h$, and for each map $\sigma:\{1,\ldots,k\}\to\{1,\ldots,h\}$ let
\[
W_\sigma = \{ x\in X  \;|\;  f(x,[g_i(x),g_{i+1}(x)])
\subset U_{\sigma(i)}, \mbox{ for all }i =1,\dots,k \}.
\]
The sets $\{W_\sigma\}_\sigma$ are a definable finite open covering of $X$. Therefore, it suffices to prove that for each $\sigma$ there is an extended lifting function for $f$ over $W_\sigma$.  Fix such a $\sigma$ and  denote $U_{\sigma(i)}$ by $U_i^\sigma$, for each $i=1,\dots, k$.  By hypothesis $p_{|p^{-1}(U^\sigma_i)}:p^{-1}(U^\sigma_i)\to U^\sigma_i  $ is a definable fibration, for each $i=1,\dots, k$. Hence by lemmata \ref{fibifflf} and \ref{lfiffelf}, for each $i$, each definable $Y$ and each definable continuous $h: Y\times I \to  U^\sigma_i  $ there is an extended lifting function for $h$ over $Y$. Therefore,  for every $i$ for every  definable $Z$ and every  definable continuous $g: Z\times I \to  B $ there is an extended lifting function for $g$ over $\{z\in Z\;|\; g(z,I)\subset U^\sigma_i \}$. That is, we are under the hypothesis of lemma \ref{piecetogether}, and consequently there is an extended lifting
function for $f$ over $W_\sigma$, as required.
\ep

\begin{proof}[Proof of theorem \ref{mainfib}]
By theorem \ref{fibtheo}, observing that trivial bundles are fibrations.
\end{proof}

\section{Homotopy of definable groups}
\begin{definition}
A {\em definable $H$-space} is a definable pointed space $(X,x_0)$ equipped
with a definable continuous map $\mu\,:\,X\times X\to X$ such that both $\mu(-,x_0)$
and $\mu(x_0,-)$ are definably homotopic to the identity.
\end{definition}

It is clear that a definable $H$-space definable over $\R$ is, in particular, an
$H$-space (in the classical sense).
\bt\label{hg} Let G be a definable group. Then  $\pi_n(G)$ is finitely generated for each $n>0$.\et
\bp
We can assume, without loss of generality, that $G$ is definably connected
and bounded.  By the triangulation theorem we can identify
$G$ with the realization $|K'|(M)$ in $M$
of a  finite simplicial complex $K'$ with rational
vertexes, one of which is the group identity $e$.  Now,  let $K$ be a
closed subcomplex of $K'$ such that $K(M)$ is a semialgebraic deformation
retract of $|K'|(M)$ containing $e$. The multiplication on $|K'|(M)$
induces a map $\nu :|K|(M)\times|K|(M)\to |K|(M)$ which gives to
$(|K|(M),e)$ the structure of a  definable $H$-space. Moreover, by
\cite[corollary 3.6]{08pBaO}, the map $\nu$ is definably homotopic 
to a semialgebraic continuous map $$\mu\colon  |K|(M)\times |K|(M)\to |K|(M)$$ definable without
parameters. Of course, $\mu$ induces again a definable $H$-space
structure over $(|K|(M),e)$. Moreover, again by \cite{08pBaO}, the semialgebraic maps  defined without parameters $\mu(-,e)$ and $id_{|K|(M)}$, which are {\em definably} homotopic, are also semialgebraically homotopic with an homotopy defined without parameters; similarly for $\mu(e,-)$ and $id_{|K|(M)}$.
\begin{claim}
Let $|K|(\R)$ denote the realization of $K$ in $\R$, then
$\pi_n(|K|(\R))$ is finitely generated for each $n\ge 1$.
\end{claim}
\bp
$(|K|(\R),e)$ is endowed with a definable $H$-space structure by $\mu$, or
more precisely by the function defined by the same formula as
$\mu$ interpreted in $\R$.
Therefore $(|K|(\R),e)$ is an $H$-space.
Notice that every path-connected $H$-space is {\em simple} ({\em i.e.} its fundamental group acts trivially on
all homotopy groups, see \cite[chapter 7 theorem 3.9]{66S}), and
for simple spaces the homotopy groups are all finitely generated if  the
homology groups are so (\cite[chapter XIII corollary 7.14]{78W}),
hence we have the claim since $|K|(\R)$ is a finite simplicial
complex. 
\ep
Now, by \cite[corollary 4.4]{08pBaO}, $\topi{n}{|K|(M)}$ is isomorphic to
$\topi{n}{|K|(\R)}$, hence it is a finitely generated abelian group.  Since $\topi{n}{|K|(M)}\cong\topi{n}{G}$, we get the result.
\ep

\bc\label{hgab} Let G be a definable abelian group. Then  $\pi_n(G)=0$, for each $n\geq2$.\ec
\bp  By theorem \ref{hg}, $\pi_n(G)$ is a finitely generated group; moreover,
we know that it is abelian for each $n\geq2$.
Since a finitely generated abelian group is divisible if and only if it is
trivial, it suffices to show that $\topi{n}{G}$ is
divisible for each $n>1$.  We may assume $G$ is definably connected. By
\cite[corollary 2.12]{04EO} the maps
\[
p_k\,:\,G\ni x\to kx\in G
\]
with $k\in\NP$ are definable covering maps,
and hence, by
\cite[corollary 4.11]{08pBaO} they induce isomorphisms on the higher homotopy groups.
As a consequence, for each $[\alpha]\in\topi{n}{G}$ and for each
$k\in\NP$ we have some
$[\beta]\in\topi{n}{G}$ such that $[p_k\circ\beta]=[\alpha]$, and
we can conclude that $\topi{n}{G}$ is divisible observing that
$[p_k\circ\beta]=k[\beta]$.
\ep

Let us denote by $\mathbb{T}^d(M)$ the $d$-dimensional torus defined as
the subset $[0,1)^d$ of $M^d$ with the sum operation modulo $1$.

\begin{theorem}\label{th2}
Let $G$ be a definably connected definably compact  $d$-dimensional abelian group.
Then $G$
is definably homotopy equivalent to $\mathbb{T}^d(M)=[0,1)^d$, {\em i.e.}, there are definable continuous maps $g\colon G\to \mathbb{T}^d(M)$ and $f\colon \mathbb{T}^d(M)\to G$, such that $f\circ g$  and $g\circ f$ are definably homotopic to the identity.
\end{theorem}
\begin{proof}
Consider the following map, defined in the proof of
\cite[lemma 4.3]{08pOP}:
\[
f\,:\,\mathbb{T}^d(M)
\ni (t_1,\ldots,t_d)
\mapsto
\gamma_1(t_1)+\ldots+\gamma_d(t_d)
\in G,
\]
where $\ominpi{1}{G}$ is the abelian group freely generated by
$[\gamma_1],\ldots,[\gamma_d]$.  Then, clearly   $f$ induces an isomorphism between  $\ominpi{1}{\mathbb{T}^d(M)}$ and
$\ominpi{1}{G}$.  To see the latter, consider  for  each $i=1,\dots,d$, the loop $\delta_i\colon [0,1]\to \mathbb{T}^d(M)$ defined by $\delta_i(t)= (0,\dots,\stackrel{i}{t},\dots,0)$ if $0\leq t<1$ and $\delta_i(1)=\delta_i(0)$. Hence, $\ominpi{1}{\mathbb{T}^d(M)}=\langle [\delta_1], \dots [\delta_d]\rangle$  and so the map $f$ induces on the fundamental groups is $f_*([\delta_i])=[\gamma_i]$, for each $i=1,\dots,d$, which is an isomorphism.

Since by theorem \ref{hgab} all the higher homotopy groups of
both $G$ and $\mathbb{T}^d(M)$ are $0$, $f$ induces  an isomorphism on them
as well (in a trivial way), hence, by the o-minimal version of Whitehead
theorem proved in \cite[theorem 5.6]{08pBaO}, $f$ is an homotopy equivalence.
\end{proof}

By \cite[theorem 1.1]{04EO}, and applying duality, we have
$\H{i}{G}{\Q}\cong \Q^{\binom{d}{i}}$ for each $i\ge0$. Here we improve
that result by proving the following.

\begin{corollary}
Let $G$ be a definably connected definably compact $d$-dimensional abelian group. Then the o-minimal homology group 
$\H{i}{G}{\Z}\cong \Z^{\binom{d}{i}}$, for each $i\ge 0$.
\end{corollary}
\begin{proof}
By theorem \ref{th2}, for each $i\ge 0$ we have that $\H{i}{G}{\Z}\cong\H{i}{\mathbb{T}^d(M)}{\Z}$, and by
\cite[proposition 3.2]{02BO} 
$\H{i}{\mathbb{T}^d(M)}{\Z}$ is isomorphic to
$\H{i}{\mathbb{T}^d(\R)}{\Z}$, which in turn is  isomorphic to $\Z^{\binom{d}{i}}$.
\end{proof}

Next we proceed to study the o-minimal homotopy groups of definably
compact  (noncommutative) groups. Towards this aim we first prove some
results concerning   homogeneous spaces of the type $G/H$, where $G$ is a
definable  group (not necessary definably compact) and $H$ is a definable
subgroup of $G$.  Notice that by \cite[corollary 2.14]{00PPS} such
$G/H$ can be equipped with a definable manifold topology so that the
canonical action of $G$ on $G/H$ is continuous. Moreover, by
\cite[theorem 4.3]{06B}, this topology coincides with the quotient topology
(inherited by that of $G$), which in turn coincides with the definable
group topology -- provided $H$ is normal -- of $G/H$. Notice also that $H$
being closed in $G$, the coset space $G/H$ is a regular definable space,
and hence we can consider its definable manifold topology induced by that
of the ambient space.

\bfa Let $F$ be the functor  from the category of  definably compact  groups to the category of compact Lie groups which sends  $G$ to $G/G^{00}$.
\begin{enumerate}
\item[(i)]  $F$ preserves dimension  and connectedness (each concept in its category).

\item[(ii)] $F$ is exact {\em i.e.} transforms  exact sequences in exact sequences. 
\end{enumerate} 
\efa
\bp By \cite[theorem 8.1]{08HPP} and \cite[theorem 5.2]{07B}, respectively.
\ep 

Our next result says that the functor $F$ also preserves the homotopy
groups.

\bt\label{Lietransfer} Let $G$ be a definably connected definably compact group.
Then $\pi_n(G) \cong \pi_n(G/G^{00})$ for all $n$.
\et
\bp
{\it Case $G$ abelian.} Let $d$ be the dimension of $G$.
By \cite[theorem 1.1]{04EO} and corollary \ref{hgab}  above, we have $\pi_1(G) \cong\Z^d$
and $\pi_n(G)=0$ for each $n>1$, respectively. On the other hand, since
$G/G^{00}$ is an abelian  compact Lie group of dimension $d$, we also have
$\pi_1(G/G^{00})\cong\Z^d$  and for $n>1$,  $\pi_n(G/G^{00})=0$. 

{\it Case $n=1$.} First note that  by the o-minimal
Poincar\'e-Hurewicz theorem (\cite[theorem 5.1]{04EO}), and the corresponding  classical result, it suffices to prove that $H_1(G) = H_1(G/G^{00})$.
On the other hand, by \cite[corollary 5.2]{07pB} and
\cite[remark 7.3]{07B}, the singular cohomology groups  $H^1(G; L)$ and $H^1(G/G^{00}; L)$  are isomorphic for any coefficient group L. Hence, by the universal coefficient theorem (UCT) for cohomology, we are done in this case. The details are as follows: The UCT says that for any chain complex $\mathcal{C}$, for any abelian group $L$ and for any $n> 0$,  if $H_{n-1}(\mathcal{C})$ is free then $H^n(\mathcal{C};L) \cong Hom(H_n(\mathcal{C}); L)$.  Applying the UCT for $n=1$ to the (honest) chain complex associated to the definable singular simplexes on our definable group $G$ we get $H^1(G,L) \cong Hom(H_1(G), L)$.  Hence, by the corresponding result for the singular chain complex of $G/G^{00}$, $Hom(H_1(G);L)\cong Hom(H_1(G/G^{00}); L)$ for any abelian group $L$. It follows that $H_1(G) \cong H_1(G/G^{00})$, since both groups $H_1(G)$ and $H_1(G/G^{00})$ are abelian and finitely generated.

Before we prove the result in the remaining cases we have the following.
\begin{claim} \label{claim2}
Let $G$ and $H$ be two definably connected definably compact groups. Suppose  $G$ is a finite extension of $H$, Then,  for every $n>1$, if  $\pi_n(H) \cong \pi_n(H/H^{00})$ then $\pi_n(G) \cong \pi_n(G/G^{00})$. 
\end{claim}
\bp
By \cite[proposition 2.11]{04EO}, the onto homomorphism  $G\to H$ with
finite kernel is a definable covering map.  Then, $\pi_n(G) \cong
\pi_n(H)$ for any $n>1$ by \cite[corollary 4.11]{08pBaO}.
By preservation of exactness,  the induced map from $G/G^{00}$ to $H/H^{00}$ is also an onto homomorphism with finite kernel, hence a covering map and so  $\pi_n(G/G^{00}) \cong \pi_n(H/H^{00})$, for any $n>1$.
\ep

{\it Case $G$  definably semisimple and $n>1$.} Since the centre of $G$ is
finite, $G$ is  a finite extension of 
the definably semisimple centreless group
$G/Z(G)$. By claim \ref{claim2},  we may consider $G$ is centreless. On the
other hand, we may assume by results  in  \cite[theorem 4.1]{00PPS} and  \cite[proof of theorem 5.1(3)]{02PPS} (see  \cite[theorems 5.3, 4.2]{08O}) that $G=G(M)$ is a semialgebraic group over the real algebraic numbers. By
\cite[corollary 4.4]{08pBaO},  $\pi_n(G(M)) \cong\pi_n(G(\R))$, for every
$n$. But $G(\R)$ is $G/G^{00}$ by  the proof of \cite[proposition 3.6]{04P}.

{\it General case.} Let $Z$ be the centre of $G$. Then, by
\cite[corollary 5.4]{99PS},  the group $G/Z$ is
definably semisimple, hence the result holds for $G/Z$.  By Corollary
\ref{fibration} the projection map $G\to G/Z$ is a definable fibration.
Therefore,  by \cite[theorem 4.9]{08pBaO}, for each $n\geq 2$, the o-minimal homotopy groups $\pi_n(G,Z)$
and $\pi_n(G/Z)$ are isomorphic; and, hence, the o-minimal
homotopy sequence of the pair  $(G,Z)$ is the following long exact
sequence (see \cite[section 4]{08pBaO})
\[\cdots\to\pi_{n+1}(G/Z)\to\pi_n(Z)\to\pi_n(G)\to\pi_n(G/Z)\to\pi_{n-1}(Z)\to\cdots\]
On the other hand, using the exactness of the    
functor to the Lie category we have that $F(Z)$ is a closed normal
subgroup of $F(G)$, so the projection map
$F(G)\to F(G)/F(Z)(\cong F(G/Z))$ is a fibration and hence we have the 
exact sequence
\[\cdots\to\pi_{n+1}(F(G/Z))\to\pi_n(F(Z))\to\pi_n(F(G))\to\pi_n(F(G/Z))\to\cdots\]
Since $Z$ and $F(Z)$ are abelian, we have $\pi_n(Z) = 0$ and $\pi_n(F(Z)) = 0$ for all $n\geq 2$. So for $n\geq 3$, $\pi_n(G) \cong \pi_n(G/Z)$ and $\pi_n(F(G)) \cong \pi_n(F(G/Z))$. By the previous case $\pi_n(F(G/Z)) \cong \pi_n(G/Z)$. For $n=2$, recall that  the second homotopy group of a  compact Lie group is trivial, so $\pi_2(F(G/Z)) = 0$, and therefore also $\pi_2(G/Z) = 0$ (by the semisimple case).  Since $\pi_2(Z)=0$, it follows that  also $\pi_2(G)=0$. 
\ep 

\br Let $k$ be a characteristic zero field. Then $H_n(G;k)\cong H_n(G/G^{00};k)$, for each $n$.
\er
\bp By a similar argument as the one done in the above proof with the UCT.  Indeed, we can use the following facts:
 (i) the singular cohomology groups  $H^1(G; L)$ and $H^1(G/G^{00}; L)$
are isomorphic for any coefficient group L (\cite[corollary 5.2]{07pB});
(ii) the UCT;
(iii) the fact that $H_n(G;k)$ are free, and (iv) $H_n(G;k)$ are
finite-dimensional  $k$-vector spaces.
\ep

\section{Universal cover}

\bfa\label{unicov} Given a definably connected  group $G$
 there is a {\em
universal covering homomorphism  $p:\uc{G}\to G$} such that:
\begin{enumerate}
\item[(i)] $p$ is a locally definable covering map with $ker\,p=\topi{1}{G}$;

\item[(ii)] $\uc{G}$ is a locally definable regular paracompact group;

\item[(iii)] $\uc{G}$ is  simply connected, and 

\item[(iv)] $\uc{G}$ is unique up to locally definable isomorphisms (\,{\em i.e.}, given
$p_1\,:\,\uc{G}_1\to G$ satisfying the above conditions, $\uc{G}_1$ is
isomorphic to $\uc{G}$ via a locally definable isomorphism). 
\end{enumerate}
\efa
\bp See \cite{07EE} and  \cite[Fact 6.13]{08pBaO2}.
\ep
\bprop\label{ucss} Let $G$ be a definably compact  semisimple group. Then,
the universal covering group of $G$ is a definably compact definable
group.\eprop
\bp
It suffices to prove that the fundamental group of $G$ is finite. Indeed,
by  \cite[theorem 3.11]{07EE}
$\uc{G}$ is  a locally definable group and by the above
fact $\uc{G}/\topi{1}{G}\cong G$.

If $G$ is semisimple and centreless, then we may suppose that $G=G(M)$ is
a semialgebraic group, hence $\topi{1}{G(M)}\cong\topi{1}{G(\R)}$ by
\cite[corollary 2.10]{02BO}. But $G(\R)$ is a semisimple compact Lie
group and hence, by a theorem of Weyl, it has a finite fundamental group (see {\em e.g.},  \cite[theorem 23.1]{04Bu}).

If $G$ is semisimple then  $Z(G)$ is finite and $G/Z(G)$ is semisimple
and centreless. By \cite[proposition 2.11]{04EO} the projection
$p:G\to G/Z(G)$ is a definable covering map, and by
\cite[corollary 2.8]{04EO} the map that $p$ induces between the fundamental groups is
injective.  We have already proved that $\topi{1}{G/Z(G)}$ is finite,
and hence so it is $\topi{1}{G}$.
\ep 

\bt Let  $G$ be a definably compact abelian group. Then the universal covering group  $\uc{G}$ is 
contractible in the category of locally definable spaces.
\et
\bp
We first prove that  all the homotopy groups  of $\uc{G}$ are trivial.
Indeed, by  Fact \ref{unicov}, $\pi_1(\uc{G})=0$ and  $p:\uc{G}\to G$ is a
covering map.  By \cite[corollary 6.12]{08pBaO2} $p$
induces isomorphism on the higher homotopy groups,
hence,
by theorem \ref{hgab}, we also have that $\pi_n(\uc{G})$ is trivial for
each $n\ge 1$.  Finally, by the o-minimal version of Whitehead theorem
(see \cite[theorem 6.16]{08pBaO2})
$\uc{G}$ is locally definably contractible. 
\ep

\end{document}